\documentclass[12pt,sn-mathphys-ay]{sn-jnl}

\usepackage{graphicx}%
\usepackage{multirow}%
\usepackage{amsmath,amssymb,amsfonts,}%
\usepackage{mathtools}
\usepackage{amsthm}%
\usepackage{mathrsfs}%
\usepackage[title]{appendix}%
\usepackage{xcolor}%
\usepackage{textcomp}%
\usepackage{manyfoot}%
\usepackage{booktabs}%
\usepackage{algorithm}%
\usepackage{algorithmicx}%
\usepackage{algpseudocode}%
\usepackage{listings}%
\usepackage{hyperref}
\usepackage{bm}
\usepackage{dsfont}
\usepackage[labelfont=bf,labelsep=space]{caption}
\usepackage{enumitem}
\theoremstyle{thmstyleone}%
\newtheorem{theorem}{Theorem}
\newtheorem{theoremletter}{Theorem}

\newtheorem{corollary}[theorem]{Corollary}%
\newtheorem{lemmaletter}[theoremletter]{Lemma}

\theoremstyle{thmstyletwo}%
\newtheorem{remark}{Remark}%

\theoremstyle{thmstylethree}%
\begin{document}
\newcommand{\R}{\mathds{R}}
\newcommand{\So}{\mathcal{S}}
\newcommand{\M}{M^{+}(\R^n)}
\newcommand{\wpp}{\mathbf{W}_{0}^{1,p}(\Omega)}
\newcommand{\wa}{\mathbf{W}_{\alpha,p}}
\newcommand{\wao}{\mathbf{W}_{\alpha,p}}
\newcommand{\wat}{\mathbf{W}_{\alpha,p}}
\newcommand{\wak}{\mathbf{W}_{k}}
\newcommand{\wah}{\mathbf{W}_{\frac{2k}{k+1},k+1}}
\newcommand{\ia}{\mathbf{I}_{2\alpha}}
\newcommand{\Om}{\Omega}
\newcommand{\e}{\epsilon}
\newcommand{\g}{\gamma}
\newcommand{\al}{\alpha}
\newcommand{\un}{u_{n}}
\newcommand{\vf}{\varphi}
\newcommand{\wc}{\rightharpoonup}
\newcommand{\phik}{\Phi^{k}(\Om)}
\newcommand{\la}{\lambda}
\newcommand{\Linf}{\Delta_{\infty}}
\newcommand{\dx}{\mathrm{dx}}
\newcommand{\Div}{\mathrm{div}}
\def\norma#1#2{\|#1\|_{\lower 4pt \hbox{$\scriptstyle #2$}}}
\def\L#1{L^{#1}(\Omega)}

\title{Quasi-linear elliptic equations with superlinear convection}
\author{\fnm{Genival} \sur{da Silva}\footnote{email: gdasilva@tamusa.edu, website: \url{www.gdasilvajr.com}}}
\affil{\orgdiv{Department of Mathematics}, \orgname{Texas A\&M University - San Antonio}}


\abstract{We discuss the existence and regularity of solutions to a quasi-linear elliptic equation involving a Leray-Lions operator and a convection term with superlinear growth. In particular, equations involving the p-Laplacian are covered. This paper generalizes some of the results in \cite{boc24}.}




\maketitle
\section{Introduction and main results}\label{sec1}
In this work we consider a nonlinear problem whose basic model is:
\begin{equation} \label{plap}
\begin{cases}
 -\Delta_{p}u= -\Div(h(u) E(x))+f(x) \qquad & \mbox{in } \Omega,\\
u (x) = 0 & \mbox{on }  \partial \Omega,
\end{cases}
\end{equation}
for $p\ge 2$ a given positive number and $h$ with superlinear growth.

More precisely, in this paper we analyze the existence and regularity of solutions to the more general problem:
\begin{equation} \label{main}
\begin{cases}
 -\Div(a(x,u,Du))= -\Div(h(u)E(x))+f(x) \qquad & \mbox{in } \Omega,\\
u (x) = 0 & \mbox{on }  \partial \Omega,
\end{cases}
\end{equation}
where $a(x,s,\xi)$ is a Leray-Lions operator defined in $\wpp$ with $p\ge 2$. Namely, the function $a:\Om\times\Om\times \R^{n}\to \R^{n}$ is Caratheodory, which satisfies, for a.e. $x\in\Om$, any $s\in\R$, any $\xi,\xi^{'}$, with $\xi\neq \xi^{'}$:
\begin{equation}\label{llcon}
\begin{split}
	& (a(x,s,\xi)-a(x,s,\xi^{'}))(\xi-\xi^{'})>0, \\
	& a(x,s,\xi)\xi \geq \al |\xi|^{p} \;  \mbox{for some } \, \al>0,\\
	& |a(x,s,\xi)| \leq \beta (b(x)+|s|^{p-1}+|\xi|^{p-1}).\\
\end{split}
\end{equation}
The function $f(x)$ and the vector field $E(x)$ belong to suitable Lebesgue spaces, and the function $h(x)$ is locally Lipschitz with super-linear growth, more precisely:
\begin{equation}
\label{hcon}
h\in \mathbf{W}^{1,\infty}_{loc}(\mathbb{R}),  \ \ \  h(0)=0, \ \ \ \mbox{and} \ \ \ \lim_{|s| \to  \infty} \frac{|h(s)|}{|s|}=+\infty.
\end{equation}

\vspace{0.1in}
\begin{remark}
Notice that if $a(x,s,\xi)=|\xi|^{p-2}\xi$ we recover \eqref{plap}. It may seem easier to use variational methods in this case but notice that $E$ does not necessarily have a potential function, i.e. $E(x)=Dg(x)$.
\end{remark}
\vspace{0.1in}
\begin{remark}
The fact that $h(s)$ has superlinear growth makes the problem interesting since in this case the operator $A(u)=-\Div(a(x,u,Du) - h(u)E(x))$ is not coercive, and the Leray-Lions existence theorem can not be applied directly.
\end{remark}
\vspace{0.1in}
Define 
\[
R(k)=\int_0^{k}\frac{ds}{(|h(s)|+1)^{p'}}, \quad S(t)=\int_0^t\frac{ds}{(|h(s)|+1)^{\frac{p'}{p}}}.
\]
Due to the superlinear growth of $h(s)$, the function $R(k)$ is bounded and $R(u_{n})\in\wpp$. Similar to the linear case, existence and regularity of solutions is connected to the growth of $S(t)$ at infinity. 

Depending on $h(s)$, the function $S(t)$ may have a horizontal asymptote, e.g. $h(s)=|s|s^{\theta}$, or go to infinity, e.g. $h(s)=s\log(1+|s|)$.

This work is motivated by and generalizes \cite{boc24}, where the authors analyzed existence and regularity in the linear case $$a(x,s,\xi) = M(x)\xi$$ of the above Dirichlet problem for many values of $h(s)$ and summability assumptions on the source $f$. What inspired me to write this text was the question: 

\vspace{0.1in}
\underline{\textit{Is the quasi-linear case similar to the linear case?}}
\vspace{0.1in}

Despite most of the arguments to continue being valid, some modifications and/or new ideas are needed in the quasi-linear case, specially if we don't assume strong monotonicity, uniqueness can fail. 

The case $p=1$ is not treated here, and it would be interesting to see the equivalent results in this setting and its geometric significance, specially its connection with equations of mean curvature type, see section \ref{last}.

We will also be interested in the same problem with added low order term. More precisely for a fixed $\mu>0$:
\begin{equation} \label{main2}
\begin{cases}
 -\Div(a(x,u,Du)) + \mu u= -\Div(h(u)E(x))+f(x) \qquad & \mbox{in } \Omega,\\
u (x) = 0 & \mbox{on }  \partial \Omega,
\end{cases}
\end{equation}
Similar to the linear case, the presence of $\mu u$ increases the regularity of the solution even if $a(x,s,\xi)$ is not linear, e.g. $p>2$ in \eqref{plap}. In the case we describe, the added term makes possible to omit the smallness condition usually needed in these cases.

By a weak solution of \eqref{main} or \eqref{main2} with $\mu\ge 0$, we mean a function $u\in\wpp$ such that for $f\in L^{(p^{*})'}(\Om)$, $h(u)E(x)\in [L^{p'}(\Om)]^{n}$ we have
\begin{equation}
\label{wk}
\int_{\Omega} a(x,u,Du) D \varphi + \mu \int_{\Omega} u\vf=\int_{\Omega}h(u)E(x) D \varphi  + \int_{\Omega} f \varphi \quad \forall \varphi \in  \wpp.
\end{equation}

\subsection{Existing literature and background}
When $a(x,s,\xi)=M(x)\xi$ and $h(u)=u$, the problem is well understood. The seminal work of Stampacchia \cite{sta65} introduced techniques and ideas to study this type of problem and its generalizations. 

When $p=2$, the problem is connected with static solutions of the transport equation:
\[
\partial_{t} u -\Div(M(x)Du - h(u)E(x))=f
\]
with a wide variety of applications, ranging from biology \cite{hil09} to pedestrian dynamics \cite{iuorio}.

In \cite{boc09}, Boccardo studies the linear problem when $h(u)=u$ with $f$ having low summability $s<\frac{N}{2}$, in particular the existence of bounded solutions is not guaranteed. We prove a generalization of this result in theorem \ref{uns}.

When $f$ has even lower summability, i.e. $s<(2^{*})'$, the right hand side in \eqref{main} is not in $H^{-1}(\Om)$, which causes a major obstacle to the existence. Despite this, the authors in \cite{boc89,boc92} are able to prove existence and regularity results in this setting as well. For a more modern reference, see \cite{boc13}.

The book \cite{boc13} provides a great introduction to quasi-linear elliptic equations of divergence type involving operators of Leray-Lions type. Detailing major techniques and approaches to existence and regularity of classical problems in nonlinear elliptic theory. For a more comprehensive and classical reference of elliptic equations of second order we recommend \cite{trudinger}.

\subsection{Summary of results}
The first theorem concerns the existence of solutions assuming \eqref{llcon} and \eqref{hcon}. A major difference in this scenario compared to the linear case is that uniqueness is not guaranteed unless we impose a stronger monotonicity. We prove the following:
\vspace{0.1in}
\begin{theorem}
Assuming \eqref{llcon},\eqref{hcon}, $E\in [\L r]^{N}$ with $r>N$, $f\in\L m$ with $m>\frac{N}{p}$, and suppose the following condition holds:
\begin{equation}
\lim_{|t|\to \infty}|S(t)|=\infty.
\end{equation}
Then the Dirichlet problem \eqref{main} has a bounded solution $u\in\wpp\cap L^{\infty}(\Om)$.
\end{theorem}
In case we assume a strong monotonicity, uniqueness hold. More precisely we show:
\begin{corollary}(Strong monotonicity)
Assuming \eqref{llcon},\eqref{hcon}, $E\in [\L r]^{N}$ with $r>N$, $f\in\L m$ with $m>(p^{*})'$. Suppose  
\begin{equation}\label{strm}
[a(x,s,\xi)-a(x,s',\xi')](\xi-\xi')\geq \al |\xi-\xi'|^{p}
\end{equation}
 for some $\al>0$. Then uniqueness holds. Moreover, any weak solution of \eqref{main} $u\in\wpp$ that is not bounded but 
 \begin{equation}
|h'(s)|\leq  C(|s|^{\theta}+1)\\\ \mbox{a.e. and }\\\ |v|^{\theta}|E| \in \L{p'}
\end{equation}
for some $\theta>0$, is also unique.
\end{corollary}
In the linear case $a(x,s,\xi)=M(x)\xi$ with $M$ uniformly elliptic, \eqref{strm} is just a consequence of ellipticity. 

The next result address the case when \eqref{Scon} holds but $m < \frac{N}{p}$, which is not strong enough to obtain bounded solutions.

\begin{theorem} 
Assuming \eqref{llcon}, \eqref{strongM} , take
$E\in [\L r]^N$,  with $ r>N$, and 
$f\in  \L{(p^{*})'}$. Then, there exists a unique (possibly unbounded) weak solution $u\in\wpp$ to
\begin{equation} 
\begin{cases}
-\Div(a(x,u,Du))= -\Div(u\log (e+|u|)E(x))+f(x) \qquad & \mbox{in } \Omega,\\
u (x) = 0 & \mbox{on }  \partial \Omega.
\end{cases}
\end{equation}
\end{theorem}

In case \eqref{Scon} fails, we can still prove existence of solutions with a smallness assumption:
\begin{theorem} 
Suppose \eqref{llcon},\eqref{strongM} and consider $\theta>0, m\in [(p^{*})',\frac{N}{p}]$ such that
\begin{equation}\label{tcon}
0<\frac{\theta}{s}=\frac{p-1}{N}-\frac{1}{r},\; f\in\L{m},\; E \in [\L{r}]^{n},
\end{equation}
where $s=\frac{Nm}{N-mp}=m^{\overbrace{**\ldots*}^{p\text{-times}}}$. If 
\begin{equation}
\norma{f}{m}\norma{E}{r}^{\frac{1}{\theta}}\le \frac{\theta}{\So}\left(\frac{\alpha p^{*}}{\So s(\theta +1)}\right)^{1+\frac{1}{\theta}},
\end{equation}
then the Dirichlet problem 
\begin{equation}\label{lows}
\begin{cases}
 -\Div(a(x,u,Du))= -\Div(u|u|^{\theta}E(x))+f(x) \qquad & \mbox{in } \Omega,\\
u (x) = 0 & \mbox{on }  \partial \Omega,
\end{cases}
\end{equation}
has a unique solution $u\in\wpp\cap\L{s}$.
\end{theorem}
Finally, our last result address \eqref{main2}, where no smallness assumption is needed but a lower order term ``$\mu u$'' is added.
\begin{theorem}
Suppose \eqref{llcon},\eqref{strongM} and consider $\theta<\frac{p-1}{N}, E\in [\L{r}]^{N}$ with $r=\frac{N}{p-1-\theta N}$, and $f\in \L{(p^{*})'}$. If $h(s)=s^{p-1}|s|^{\theta}$, then  \eqref{main2} has a unique solution $u\in\wpp$.
\end{theorem}
\subsection{Notation}
\begin{itemize}
\item[-] $\Om\subset \R^{n}$ is a bounded domain.
\item[-] The space $\wpp$ denotes the usual Sobolev space which is the closure of $\mathcal{C}^{\infty}_{0}(\Om)$, smooth functions with compact support, in the $p$-norm. 
\item[-] For $1<p<\infty$, the \textit{p-Laplacian} $\Delta_{p}$ is given by $-\Div (|Du|^{p-2}Du)$. 
\item[-] For $q>0$, $q'$ denotes the Holder conjugate, i.e. $\frac{1}{q}+\frac{1}{q'}=1$, and $q^{*}$ denotes the Sobolev conjugate, defined by $q^{*}=\frac{qn}{n-q}>q$, where $n$ is the dimension of the domain $\Om\subset \R$. 
\item[-] Given a function $u(x)$, we denote its positive part by $u^{+}(x)=\max(0,u(x))$.
\item[-] We will use the somewhat standard notation for Stampacchia's truncation functions (see \cite{boc13}):
\[
T_k(s)=\max\{-k,\min\{s,k\}\}, \ \ \ G_k(s)=s-T_k(s), \quad \mbox{ for } k>0.
\]
\item[-] The letter $C$ will always denote a positive constant which may vary from line to line. 
\item[-] The Lebesgue measure of a set $A\subseteq \R^{n}$ is denoted by $|A|$.
\item[-] The symbol $\rightharpoonup$ denotes weak convergence. 
\end{itemize}
\section{Proof of the results}
\subsection{Existing Lemmata}
We will need the following lemmata from the book \cite{boc13}:
\vspace{0.1in}
\begin{lemmaletter}\label{lemA}\cite[Lem~6.2]{boc13}
Let $f\in L^{1}$, $g(k)=\int_{\Om} |G_{k}(f)|$ and $A_{k}=\{ |f|>k \}$. Suppose
\[
g(k)\le \beta |A_{k}|^{\al}
\]
for some $\al>1$ and $\beta>0$. Then $f\in \L{\infty}$ and 
\[
\norma{f}{\infty}\le C\beta 
\]
for some $C=C(\al,\Om)$. 
\end{lemmaletter}
\vspace{0.1in}
\begin{lemmaletter}\label{lemB}\cite[Lem~5.9]{boc13}
Let $\un,u\in\wpp$ such that $\un\rightharpoonup u$ in $\wpp$. If
\[
\int_{\Om} [a(x,\un,D\un)-a(x,u,Du)]\cdot (D\un-Du)\to 0
\]
then $a(x,\un,D\un) \rightharpoonup a(x,u,Du)$ in $\L{p'}$.
\end{lemmaletter}
\subsection{General case with high summability of the source}
\vspace{0.1in}
\begin{theorem} \label{decay}
Assuming \eqref{llcon},\eqref{hcon}, $E\in [\L r]^{N}$ with $r>N$, $f\in\L m$ with $m>\frac{N}{p}$, and suppose the following condition holds:
\begin{equation}\label{Scon}
\lim_{|t|\to \infty}|S(t)|=\infty.
\end{equation}
Then the Dirichlet problem \eqref{main} has a bounded solution $u\in\wpp\cap L^{\infty}(\Om)$.
\end{theorem}
\begin{proof} 
Fix $n>0$ and let $h_{n}(s)=T_n(h(s)),f_{n}(x) = T_n(u(x))$ and $E_{n}(x)=T_n(E(x))$ be the vector field $E(x)$ whose components are truncated by $n$. Classical results by Leray-Lions (see \cite{boc13}) guarantee the existence of a function $u_{n}\in\wpp$ satisfying
\begin{equation}\label{twk}
\int_{\Omega} a(x,u_{n},Du_{n}) D \varphi =\int_{\Omega}h_{n}(u_{n})E_{n}(x) D \varphi  + \int_{\Omega} f_{n}(x) \varphi \quad \forall \varphi \in  \wpp.
\end{equation}
Additionally, Stampacchia's regularity gives us that $u_{n}\in\wpp\cap L^{\infty}(\Om)$, so the idea is to find estimates for $u_{n}$ independent of $n$.

Now take $\vf= G_{R(k)}(R(u^{+}_{n}))$ and use Young's inequality in equation \eqref{twk} to get:
\[
\begin{split}
\alpha\int_{\{k<\un^+\}} \frac{|D \un^+|^p}{(|h(\un^+)|+1)^{p'}}\le & \int_{\{k<\un^+\}} |h(\un^+)||E| \frac{|D \un^+|}{(|h(\un^+)|+1)^{p'}}+\int_{\{k<\un^+\}} |f| R(\un^+)\\
\le &C \int_{\{k<\un^+\}} |E(x)|^{p'}+\frac\alpha 2\int_{\{k<\un^+\}} \frac{|D \un^+|^p}{(|h(\un^+)|+1)^{p'}}+C \int_{\{k<\un^+\}}|f(x)|,
\end{split}
\] 
so
\[
\frac{\al}{2}\int_{\{k<\un^+\}} |D S(u_{n}^+)|^{p}\le C\left( \int_{\{k<\un^+\}} |E(x)|^{p'}+ \int_{\{k<\un^+\}}|f(x)|\right)
\]
By Sobolev's inequality and the fact that $\{ x \in\Om\;|\; u_{n}^+>k\} = \{ x \in\Om\;|\; S(u_{n}^+)>S(k)\}$, we obtain:
\[
\frac{\al}{2}\left(\int_{\{k<\un^+\}} |G_{S(k)}(S(\un^+))|^{p^{*}}\right)^{\frac{p}{p^{*}}}\le \So\frac{\al}{2}\int_{\{k<\un^+\}} |D S(\un^+)|^{p}
\]
Combining everything we obtain:
\[
\left(\int_{\Om} |G_{S(k)}(S(\un^+))|^{p^{*}}\right)^{\frac{p}{p^{*}}}\le C\left(\int_{\{k<\un^+\}} |E(x)|^{p'}+ \int_{\{k<\un^+\}}|f(x)|\right)
\]
By hypothesis, $E+f\in\L{m}$ with $m>\frac{N}{2}$, using Holder's inequality and lemma \eqref{lemA} we deduce that 
\[
\norma{S(\un^+)}{\infty}\le C
\]
for some $C>0$ independent of $n$. Now, condition \eqref{Scon} implies:
\[
\norma{\un^+}{\infty}\le C'
\]
again, for some $C'>0$ independent of $n$. By repeating the same argument with $\un^{-}$ instead of $\un^{+}$ we easily reach the conclusion that 
\[
\norma{\un^-}{\infty}\le C.
\]
Therefore, we have $\norma{\un}{\infty}\le C$ for some $C>0$ independent of $n$. 

Now, taking $\vf=\un$ in \eqref{twk} we easily obtain
\[
\norma{D\un}{p}^{p}\le C.
\]
We conclude that up to a subsequence, $\un\rightharpoonup u$ in $\wpp$ and $\un\to u$ strongly in $\L{q}$ for $q<p^{*}$ and a.e. in $\Om$.

It follows that we can pass the limit in the last integral of \eqref{twk} but it's not obvious that we can pass the limit in the first and second integral. For the second integral notice that given any measurable set $A\subset \Om$:
\[
\int_A
|h_n(\un)E_n(x)Dv|
\leq
\bigg(\int_{A} |h(\un)E|^{p'}\bigg)^\frac{1}{p'}
\bigg(\int_{A}|D v|^2\bigg)^\frac1p\le C \bigg(\int_{A}|D v|^2\bigg)^\frac1p.
\]
where we used the fact that $h(\un)E\in\L{p'}$. It follows by Vitali's convergence theorem that we can pass the limit in the second integral as well. 

To deal with the first integral we plan to use lemma \eqref{lemB}. Notice that
\[
\int_{\Om} [a(x,\un,D\un)-a(x,u,Du)]\cdot (D\un-Du) = \int_{\Om} h(u_{n})E(x)D(\un-u) + \int f(u_{n}-u)-\int_{\Om} a(x,u,Du)D(\un-u)
\]
which goes to $0$ as $n\to\infty$. The result then follows.
\end{proof}
\begin{corollary}(Strong monotonicity)\label{strongcor}
Assuming \eqref{llcon},\eqref{hcon}, $E\in [\L r]^{N}$ with $r>N$, $f\in\L m$ with $m>(p^{*})'$. Suppose  
\begin{equation}\label{strongM}
[a(x,s,\xi)-a(x,s',\xi')](\xi-\xi')\geq \al |\xi-\xi'|^{p}
\end{equation}
 for some $\al>0$. Then uniqueness holds. Moreover, any weak solution of \eqref{main} $u\in\wpp$ that is not bounded but 
 \begin{equation}\label{unb}
|h'(s)|\leq  C(|s|^{\theta}+1)\\\ \mbox{a.e. and }\\\ |v|^{\theta}|E| \in \L{p'}
\end{equation}
for some $\theta>0$, is also unique.
\end{corollary}
\begin{proof}
Suppose $u,v \in\wpp\cap\L{\infty}$ are two solutions. We have $\forall \varphi \in  \wpp$:
\[
\begin{split}
&\int_{\Omega} a(x,u,Du) D \varphi =\int_{\Omega}h(u)E(x) D \varphi  + \int_{\Omega} f(x) \vf \\
&\int_{\Omega} a(x,v,Dv) D \varphi =\int_{\Omega}h(v)E(x) D \varphi  + \int_{\Omega} f(x) \vf
\end{split}
\]
Take $\vf=T_{\e}(u-v)$ and subtract both equations to obtain:
\[
\int_{\Om} (a(x,u,Du)-a(x,v,Dv)) DT_{\e}(u-v)=\int_{\{|u-v|<\e\}} [h(u)-h(v)]E(x) DT_{\e}(u-v)
\]
Using Young's inequality we have:
\[
\al \int_{\Om} |DT_{\e}(u-v)|^{p}\le \frac{2}{\al p p'}\int_{\{|u-v|<\e\}} |[h(u)-h(v)]E(x)|^{p'}  + \frac{\al}{2}\int_{\Om} |DT_{\e}(u-v)|^{p}
\]
Poincare inequality gives:
 \begin{equation}\label{uvest}
C \int_{\Om} |T_{\e}(u-v)|^{p}\le \frac{2}{\al p p'}\int_{\{|u-v|<\e\}} |[h(u)-h(v)]E(x)|^{p'},
\end{equation}
and since $u,v\in \L{\infty}$, we can use the fact that $|h(u)-h(v)|\leq C |u-v|$, hence for any $k\geq \e$:
\[
 \e^{p}|\{ |u-v|>k\}| \le C \int_{\{|u-v|<\e\}} \e^{p'}|E(x)|^{p'} \le C \int_{\{|u-v|<\e\}} \e^{p}|E(x)|^{p'}
\]
Dividing by $\e^{p}$ and letting $\e\to 0$ we conclude that $|\{ |u-v|>k\}|=0$ for any $k>0$, hence $u=v$ a.e.

Now suppose, $u,v$ unbounded, in this case we can't use the Lipschitz condition directly, but \eqref{unb} give us
\[
\frac{|h(v)-h(w)|}{|v-w|}\le C\int_0^1\big(|w+t(v-w)|^{\theta}+1\big)dt\le  C(|v|^{\theta}+|w|^{\theta}+1 ),
\]
If we use this in \eqref{uvest}, we reach the same conclusion as before, namely, $u=v$ a.e.
\end{proof}
\vspace{0.1in}
\begin{remark}
In particular, if $a(x,s,\xi)=M(x)\xi$ or $a(x,s,\xi)=|\xi|^{p-2}\xi$, then $a(x,\xi)$ satisfies the strong monotonicity condition \eqref{strongM}.
\end{remark}
\vspace{0.1in}
\begin{remark}
If $p<2$, $a(x,s,\xi)$ is not necessarily strong monotone, in fact for $1<p<2$ we have:
\[
(|Du|^{p-2}Du-|Dv|^{p-2}Dv)\cdot (Du-Dv)\ge C \frac{|Du-Dv|^{2}}{(|Du|+|Dv|)^{2-p}}.
\]
\end{remark}
\vspace{0.1in}
\subsection{Lower summability of f(x) and unbounded solutions}
\vspace{0.1in}
\begin{theorem} \label{uns}
Assuming \eqref{llcon}, \eqref{strongM} , take
$E\in [\L r]^N$,  with $ r>N$, and 
$f\in  \L{(p^{*})'}$. Then, there exists a unique (possibly unbounded) weak solution $u\in\wpp$ to
\begin{equation} \label{logp}
\begin{cases}
-\Div(a(x,u,Du))= -\Div(u\log (e+|u|)E(x))+f(x) \qquad & \mbox{in } \Omega,\\
u (x) = 0 & \mbox{on }  \partial \Omega.
\end{cases}
\end{equation}
\end{theorem}
\begin{proof} 
We proceed in a similar way we did in the previous theorem. Consider the truncated equation:
\begin{equation}\label{lwk}
\int_{\Omega} a(x,u_{n},Du_{n}) D \varphi =\int_{\Omega}T_{n}(\un\log (e+|\un|))E_{n}(x) D \varphi  + \int_{\Omega} f_{n}(x) \varphi 
\end{equation}
Take $\vf=\un$ as a test function in \eqref{lwk}:
\[
\begin{split}
\al \norma{D\un}{p}^{p}&\le \int_{\Omega}|\un\log (e+|\un|))||E_{n}(x)| |D \un| +\norma{f}{(p*)'}\norma{\un}{p^{*}}\\
&\le C\int_{\Omega}\left( |\un\log (e+|\un|)||E_{n}(x)| \right)^{p'} + \frac \al 2 \norma{D\un}{p}^{p}+\So\norma{f}{(p*)'}\norma{D\un}{p}
\end{split}
\]
Fix $k>0$ and consider the first integral on the right:
\[
\int_{\Omega}\left( |\un\log (e+|\un|)||E_{n}(x)| \right)^{p'}\le \int_{\{|\un|>k\}} |\un|^{p'}|\log (e+|\un|)|^{p'}|E_{n}(x)|^{p'} + k^{p'}\log (e+|k|)^{p'}\norma{E}{p'}
\]
Let $d=\frac{1}{N}-\frac 1 r +\frac{p-2}{p}$, we take a closer look at the following integral:
\begin{equation}\label{estk}
\int_{\{|\un|>k\}} |\un|^{p'}|\log (e+|\un|)|^{p'}|E(x)|^{p'} \le \norma{\un}{p^{*}}^{p'} \norma{\log (e+|\un|)}{d}^{p'}\left(\int_{\{|\un|>k\}} |E|^{r}\right)^{\frac{p'}{r}}
\end{equation}
The only thing remaining is then estimate the middle norm on the right, namely $\norma{\log (e+|\un|)}{d}^{p'}$.

 In order to accomplish that, we need to take a test function such that the coercivity (and possibly Sobolev inequality) of $a(x,s,p)$ will lead to a bound of the norm. More precisely, we will find a bound to the expression $\norma{\log^{j} (e+|\un|)}{p^{*}}$, for any $j>1$. This can be done by choosing $\vf=\int_0^{\un}\frac{\log^{p(j-1)}(e+|s|)}{(e+|s|)^p}ds$ as a test function in \eqref{lwk} (notice that $|\vf|\le C$ in this case):
\[
\begin{split}
\frac{\al}{j^{p}}\int |D\log^{j}(e+|\un|)|^{p}&=\alpha\int |D \un|^p\frac{\log^{p(j-1)}(e+|\un|)}{(e+|\un|)^p} \le \int |E(x)||D \un| \frac{\log^{p(j-1)+1}(e+|\un|)}{(e+|\un|)^{(p-1)}} +C\norma{f}{1}
\\ \le& C\int\frac{\log^{p'[(p-1)(j-1)+1]}(e+|\un|)}{(e+|\un|)^{p'(p-2)}}|E(x)|^{p'}+\frac{\alpha}{2}\int|D \un|^p\frac{\log^{p(j-1)}(e+|\un|)}{(e+|\un|)^p}+C\norma{f}{1},\\ \le& C\int \log^{p'[(p-1)(j-1)+1]}(e+|\un|)|E(x)|^{p'}+\frac{\alpha}{2}\int|D \un|^p\frac{\log^{p(j-1)}(e+|\un|)}{(e+|\un|)^p}+C\norma{f}{1}
\end{split}
\] 
Rearranging, we have:
\[
\begin{split}
\frac{\al}{\So 2j^{p}} \norma{\log^{j}(e+|\un|)}{p*}^{p}\le \frac{\al}{2j^{p}}\int |D\log^{j}(e+|\un|)|^{p}& \le C\left(\int \log^{p'[(p-1)(j-1)+1]}(e+|\un|)|E(x)|^{p'} +\norma{f}{1}\right)
\\ \le& C\left(\int_{\{|\un|>k\}} \log^{p'[(p-1)(j-1)+1]}(e+|\un|)|E(x)|^{p'}\right)\\ &+ C\left(\log^{p'[(p-1)(j-1)+1]}(e+k)\norma{E}{p'}^{p'}+\norma{f}{1}\right)\\
\\ \le& C\left(\int_{\{|\un|>k\}} \log^{jp}(e+|\un|)|E(x)|^{p'}\right)\\ &+ C\left(\log^{p'[(p-1)(j-1)+1]}(e+k)\norma{E}{p'}^{p'}+\norma{f}{1}\right)\\
\end{split}
\] 
Using Holder inequality, this becomes:
\[
\frac{\al}{\So 2j^{p}} \norma{\log^{j}(e+|\un|)}{p*}^{p}\le C (\norma{\log^{j}(e+|\un|)}{p*}^{p}\left(\int_{\{|\un|\ge k\}} |E(x)|^{\frac{N}{p-1}}\right)^{p'}+\norma{E}{p'}^{p'}+\norma{f}{1})
\]
By Chebyshev's inequality and uniform continuity of the integral, we can make the last integral as small as we want, so we conclude that
\[
\norma{\log^{j}(e+|\un|)}{p*}^{p}\le C\left(\norma{E}{p'}^{p'}+\norma{f}{1}\right )
\]
Putting this back into \eqref{estk} and combining everything we get:
\[
\al \norma{D\un}{p}^{p}\le C\left(\norma{E}{p'}^{p'}+\norma{f}{1}\right )
\]
Hence, $\un\wc u$ in $\wpp$. Using Corollary \ref{strongcor} and proceeding in the same way we did in the proof of theorem \ref{decay} the proof is complete.
\end{proof}
\subsection{Lower summability without \eqref{Scon}}
\vspace{0.1in}
\begin{theorem} \label{clow}
Suppose \eqref{llcon},\eqref{strongM} and consider $\theta>0, m\in [(p^{*})',\frac{N}{p}]$ such that
\begin{equation}\label{tcon}
0<\frac{\theta}{s}=\frac{p-1}{N}-\frac{1}{r},\; f\in\L{m},\; E \in [\L{r}]^{n},
\end{equation}
where $s=\frac{Nm}{N-mp}=m^{\overbrace{**\ldots*}^{p\text{-times}}}$. If 
\begin{equation}\label{minc}
\norma{f}{m}\norma{E}{r}^{\frac{1}{\theta}}\le \frac{\theta}{\So}\left(\frac{\alpha p^{*}}{\So s(\theta +1)}\right)^{1+\frac{1}{\theta}},
\end{equation}
then the Dirichlet problem 
\begin{equation}\label{lows}
\begin{cases}
 -\Div(a(x,u,Du))= -\Div(u|u|^{\theta}E(x))+f(x) \qquad & \mbox{in } \Omega,\\
u (x) = 0 & \mbox{on }  \partial \Omega,
\end{cases}
\end{equation}
has a unique solution $u\in\wpp\cap\L{s}$.
\end{theorem}
\begin{proof} 
The proof is topological, using Schauder fixed-point theorem. Given any $v\in\L{s}$, let $u\in\wpp\cap\L{s}$ be the unique weak solution to 
\begin{equation}
 -\Div(a(x,u,Du))= -\Div(v|v|^{\theta}E(x))+f(x). 
\end{equation}
The correspondence $v\to T(v):=u$ defines an operator $T:\L{s}\to \L{s}$, we claim $T(v)$ has a fixed point. By Schauder fixed-point theorem it's enough to prove that $T$ is continuous, compact and has a convex invariant subspace.

\vspace{.1in}
\underline{\textbf{Claim:} $T$ is continuous.}
\vspace{.1in}

We need some estimates first in order to prove continuity. Set $\g=\frac{s}{p^{*}}$ and let $ u_{kl}=T_{k}(G_{l}(u))$. Take $\vf=\frac{1}{p(\g-1)+1}|u_{kl}|^{p(\g-1)}u_{kl}$ as a test function in \eqref{lows}. Using the given hypotheses we have:
\[
\alpha\int |D u_{kl}|^{p}|u_{kl}|^{p(\g-1)}\le \int |E(x)||v|^{1+\theta}|D u_{kl}||u_{kl}|^{p(\g-1)}+\frac{1}{p(\g-1)+1}\int |f(x)||u_{kl}|^{p(\g-1)+1}.
\]
Using \eqref{tcon} and Holder inequality on the first integral on the right we deduce:
\[
\begin{split}
\int |E(x)||v|^{1+\theta}|D u_{kl}||u_{kl}|^{p(\g-1)} &  \\ 
\le C_{E}&\norma{v}{s}^{1+\theta}\norma{|D u_{kl}||u_{kl}|^{\g-1}}{p}\norma{u_{kl}}{s}^{(p-1)(\g-1)}.
\end{split}
\]
where $C_{E}=\left(\int_{\{|u|>l\}} |E|^{r}\right)^{\frac{1}{r}}$. Using the fact that $m\in [(p^{*})',\frac{N}{p}]$ and Sobolev's inequality, the second integral becomes:
\[
\begin{split}
\int |f(x)||u_{kl}|^{p(\g-1)+1}&\le C_{f} \norma{u_{kl}}{s}^{(p-1)(\g-1)} \left(\int |u_{kl}|^{p^{*}\gamma}\right)^{\frac{1}{p^{*}}}\\
&\le \So\g C_{f} \norma{u_{kl}}{s}^{(p-1)(\g-1)}\left( \int |D u_{kl}|^{p}|u_{kl}|^{p(\g-1)} \right)^{\frac{1}{p}}\\
&=\So\g C_{f} \norma{u_{kl}}{s}^{(p-1)(\g-1)}\norma{|D u_{kl}||u_{kl}|^{\g-1}}{p}
\end{split}
\]
where $C_{f}=\left(\int_{\{|u|>l\}} |f|^{m}\right)^{\frac{1}{m}}$. Combining everything together we have:
\[
\alpha\norma{D|u_{kl}|^{\g}}{p}^{p}\le \norma{u_{kl}}{s}^{(p-1)(\g-1)}\norma{|D u_{kl}||u_{kl}|^{\g-1}}{p}\left(\frac{\So\g}{p(\g-1)+1} C_{f} + C_{E}\norma{v}{s}^{1+\theta}\right)
\]
Using Sobolev inequality again:
\[
\frac{\alpha}{\So\g} \norma{u_{kl}}{s}^{(p-1)s}\le \norma{u_{kl}}{s}^{(p-1)(\g-1)} \left(\frac{\So\g}{p(\g-1)+1} C_{f} + C_{E}\norma{v}{s}^{1+\theta}\right)
\]
Finally, we obtain:
\begin{equation}\label{up2}
\norma{u_{kl}}{s}\le \frac{\So\g}{\alpha} \left(\frac{\So\g}{p(\g-1)+1} C_{f} + C_{E}\norma{v}{s}^{1+\theta}\right),
\end{equation}
using the fact that $p(\g-1)+1> \g$ since $p>1$, this becomes 
\begin{equation}\label{me}
\norma{u}{s}\le \frac{\So^{2}s}{\al p^{*}}\norma{f}{m} +\frac{\So s}{p^{*}\alpha} \norma{E}{r}\norma{v}{s}^{1+\theta},
\end{equation}
if we let $k\to\infty$ and $l\to 0$.

Suppose $v_{n}\to v$ in $\L{s}$, we claim $\un:= T(v_{n})\to u:= T(v)$ in $\L{s}$. Indeed, by \eqref{me} we have:
\begin{equation}
 \norma{u_{n}}{s}\le C
\end{equation}
Using $\vf=\un-u$ as a test function it's easy to see that $\un\to u$ a.e. By the dominated convergence theorem, $\un\to u$ in $\L{s}$, and $T(v)$ is continuous. 

\vspace{.1in}
\underline{\textbf{Claim:} $T$ is compact.}
\vspace{.1in}

Suppose $v_{n}\wc v$ in $\L{s}$, the claim is equivalent to say that $\un=T(v_{n})\to u=T(v)$ in $\L{s}$, up to a subsequence

Choose $\vf=\un$ as a test function in the weak formulation of the equation satisfied by $\un$ to obtain:
\[
\begin{split}
\alpha\norma{D\un}{p}^{p}\le& \int |v_n|^{1+\theta}|E||Du_n|+\int |f||\un|\\
\le & C (\norma{v_n}{s}^{\theta+1}\norma{E}{r}\norma{Du_{n}}{p}+\norma{f}{(p^{*})'}\norma{Du_{n}}{p})
\end{split}
\]
So $\un$ is bounded in $\wpp$, hence up to a subsequence $\un\to \xi$ in $\L{q}$ for $q<p^{*}$. But since $v_{n}\wc v$, we must have $\xi=u=T(v)$, and up to a subsequence $\un\to u$ a.e.

Now, take $\Sigma \subseteq \Om$ a measurable set, we claim $\int_{\Sigma} |\un|^{s}$ is equi-integrable. Indeed, 
\[
\int_{\Sigma}|\un|^{s}\le \int_{\Sigma}|T_l(\un)|^{s}+\int_{\Sigma} |G_l(\un)|^{s}\le l^{s}|\Sigma|+\int_{\Sigma} |G_l(\un)|^{s},
\]
Notice that by \eqref{up2}, $\int_{\Sigma} |G_l(\un)|^{s}$ goes to 0 uniformly as $l\to\infty$, hence the right hand side goes to 0 when $|\Sigma|\to 0$. The result then follows by Vitali Convergence theorem.

\vspace{.1in}
\underline{\textbf{Claim:} $T$ has a closed convex invariant subspace}.
\vspace{.1in}

Set $a=\frac{\So^{2}s}{\al p^{*}}\norma{f}{m}$, $b=\frac{\So s}{p^{*}\alpha} \norma{E}{r}$, and define $$R=\left(\frac{1}{b(\theta+1)}\right)^{\frac{1}{\theta}}.$$
Notice that if $a\le R \frac{\theta}{\theta +1}$ and $\norma{v}{s}\le R$ then $\norma{u}{s}\le a +b\norma{v}{s}^{1+\theta}\le R$, and it follows that the closed ball $B_{R}(0)$ of radius $R$ is invariant. In other words, if $$\norma{f}{m}\norma{E}{r}^{\frac{1}{\theta}}\le \frac{\theta}{\So}\left(\frac{\alpha p^{*}}{\So s(\theta +1)}\right)^{1+\frac{1}{\theta}},$$
then $T$ has a closed convex invariant subspace.

This concludes the proof of existence, uniqueness follow from the strong monotonicity.
\end{proof}

\subsection{No smallness condition if $\mu u$ is added}
\vspace{.1in}
\begin{theorem}
Suppose \eqref{llcon},\eqref{strongM} and consider $\theta<\frac{p-1}{N}, E\in [\L{r}]^{N}$ with $r=\frac{N}{p-1-\theta N}$, and $f\in \L{(p^{*})'}$. If $h(s)=s^{p-1}|s|^{\theta}$, then  \eqref{main2} has a unique solution $u\in\wpp$.
\end{theorem}
\begin{proof} 
As before, we can truncate the equation to the form:
\begin{equation}\label{wk2}
-\Div(a(x,\un,D\un)) + \mu \un= -\Div(T_{n}(\un^{p-1}|\un|^{\theta}) E_{n}(x))+f_{n}(x) 
\end{equation}
Notice that the left hand side is still coercive, since $\mu>0$, hence there is a unique solution $u_{n}\in\wpp\cap\L{\infty}$ for each $n$. Take $\vf=T_{k}(\un)$ as a test function in \eqref{wk2}, we have:
\[
 \alpha\int_{\Om} |D T_{k}(\un)|^p+\mu\int_{\Om} u_n T_{k}(u_n)\le \int_{\Om}k^{\theta+p-1} |E(x)||D T_{k}(u_n) |+ k\norma{f}{1}.
\]
We estimate the first integral on the right using Young inequality:
\[
\int_{\Om} k^{\theta+p-1} |E(x)||D T_{k}(u_n) |+ k\norma{f}{1}\le  \left( \frac{2p k^{(\theta+p-1)p'} }{\alpha p'}\norma{E}{p'} + \frac{\alpha}{2}\norma{DT_{k}(u_n)}{p}\right)+ k\norma{f}{1},
\]
so we get
\begin{equation}\label{tke}
\mu\int_{\Om} u_n T_{k}(u_n) \le \frac{\alpha}{2}\int_{\Om} |D T_{k}(\un)|^p+\mu\int_{\Om} u_n T_{k}(u_n)\le \left( \frac{k^{\theta+p-1}}{p'}\norma{E}{p'}\right)+ k\norma{f}{1},
\end{equation}
dividing everything by $k$ and letting $k\to 0$ we get
\[
\mu\norma{u_n}{1} \le \norma{f}{1},
\]
and by Chebyshev's inequality, for any $t>0$:
\begin{equation}\label{crux}
|\{x\in\Om\;|\; |\un| \ge t\}|\leq \frac{1}{t}\int_{\Om} |u_n|  \le \frac{\mu}{t}\norma{f}{1},
\end{equation}
Now, take $\vf=G_{k}(\un)$ as a test function in \eqref{wk2} to obtain:
\[
\begin{split}
 \alpha\int_{\Om} |D G_{k}(\un)|^p \le& 2^{\theta+p-1}\int_{\Om} |G_k(u_n)|^{\theta+p-1} |E(x)||D G_k(u_n)|\\
+& (2k)^{\theta+p-1}\int_{\Om}|E(x)||D G_k(u_n)|+\int_{\Om} |f||G_{{k}}(u_n)|,
\end{split}
\]
Using Holder's inequality, we can estimate the first integral to the right:
\[
\begin{split}
2^{\theta+p-1} \int_{|u_n|\ge k} |G_k(u_n)|^{p-1}& |G_k(u_n)|^{\theta}|E(x)||D G_k(u_n)|\le \\&   \le2^{\theta+p-1}\norma{G_k(u_n)}{p^*}^{p-1}\norma{G_k(u_n)}{1}^{\theta}\left(\int_{|u_n|\ge k} |E(x)|^r\right)^{\frac1r}\norma{DG_k(u_n)}{p}\\ & \le 2^{\theta+p-1}\So \mu^{-1}\norma{f}{1}\left(\int_{|u_n|\ge k} |E(x)|^r\right)^{\frac1r}\norma{DG_k(u_n)}{p}^{p}
\end{split}
\]
By \eqref{crux}, we can choose $k=k_{0}$ such that:
$$2^{\theta+p-1}\So \mu^{-1}\norma{f}{1}\left(\int_{|u_n|\ge k_{0}} |E(x)|^r\right)^{\frac1r}\le \frac{\al}{4}$$
which implies:
$$2^{\theta+p-1} \int_{|u_n|\ge k_{0}} |G_k(u_n)|^{p-1} |G_k(u_n)|^{\theta}|E(x)||D G_k(u_n)|\le \frac{\al}{4}\norma{DG_k(u_n)}{p}^{p}.$$
The remaining terms on the right can be dealt with by applying Sobolev's and Young's inequalities:
\[
(2k)^{\theta+p-1}\int_{\Om}|E(x)||D G_k(u_n)|+\int_{\Om} |f||G_{{k}}(u_n)|\le C(\norma{E}{p'}+\norma{f}{(p^{*})'})+\frac{\al}{4}\norma{DG_k(u_n)}{p}^{p}
\]
Combining everything together, we conclude that 
\[
\frac{\al}{2}\norma{DG_k(u_n)}{p}^{p} \le C(\norma{E}{p'}+\norma{f}{(p^{*})'}).
\]
Moreover, by \eqref{tke}:
\[
\frac{\al}{2}\norma{DT_k(u_n)}{p}^{p} \le C(\norma{E}{p'}+\norma{f}{1}).
\]
Since $\un=T_{k}(\un) + G_{k}(\un)$, it follows that $\norma{D\un}{p}\leq C$ for some $C$ independent of $n$, and $\un\wc u$ in $\wpp$ for some $u\in\wpp$. By using the same arguments as we did in the end of the proof of theorem \ref{decay} we can conclude that $u$ is a solution.

For the uniqueness, notice that if $r=\frac{N}{p-1-\theta N}$ then
\[
\int_{\Omega} |\un|^{(\theta+p-1)p'}|E|^{p'}\le\norma{\un}{p^{*}}^{(p-1)p'}\norma{\un}{1}^{\theta p'}\norma{E}{r}^{p'},
\]
by Corollary \ref{strongcor}, the proof is complete.
\end{proof}
\section{Concluding remarks and open questions}\label{last}
\vspace{0.1in}
The natural generalization of the problem treated in this manuscript  would be a fully non-linear equation with superlinear convection, i.e. $F(x,u,Du,D^{2}u)=h(u)E+f$. Although the right hand side of \eqref{main} can be analyzed in a similar way as we did here, new ideas would be needed to obtain estimates, namely, methods involving viscosity solutions. 
\vspace{0.1in}
\begin{enumerate}
	\item \textbf{The case p=1} In these notes we have assumed $p\geq 2$ because some of the estimates fail when $p<2$. A natural question would be if and under what circumstances we have existence, and what type of regularity should we expect. The particular case of $p=1$ is very rich and interesting on its own since it is a generalization of an equation of mean curvature type, unexpected outcomes could occur in this case and results with a geometric flavor could be obtained.
	\vspace{0.1in}
\item \textbf{Fully nonlinear generalizations} As mentioned above, the following problem seems to be interesting but can't be handle with methods shown here. Significantly new ideas would be needed, even if $E=Dg$ for some $g$. Here, $F$ degenerate elliptic with some type p-growth/p-coercivity.
\[
F(x,u,Du,D^{2}u)=h(u)E+f.
\]
Classical results usually assume $F$ convex and uniformly elliptic. It would be interesting to see if those results can be generalized for a more singular $F$.
\item \textbf{Radially symmetric solutions} Is \eqref{minc} sharp when $p>2$? In particular, if $a(x,s,\xi)=|\xi|^{p-2}\xi$ and $u(x)=\hat{u}(|x|)$, is it possible to obtain $E$ and $f$ such that \eqref{minc} depends only on $|x|$? Such result would be a generalization of \cite[Prop.~1.4]{boc24}.
\end{enumerate}

\subsection*{Acknowledgements}

\vspace{.1in}
\noindent\textbf{Data Availibility Statement} This manuscript has no associated data.

\vspace{.1in}
\noindent\textbf{Funding} No funding was received for conducting this study.

\vspace{.1in}
\noindent\textbf{Conflict of interest} We certify that the submission is original work and is not under review at any other publication.

\vspace{.1in}
\noindent\textbf{Ethics approval} No ethical approval is required.
 
\bibliography{sn-article}


\begin{thebibliography}{9}
\ifx \bisbn   \undefined \def \bisbn  #1{ISBN #1}\fi
\ifx \binits  \undefined \def \binits#1{#1}\fi
\ifx \bauthor  \undefined \def \bauthor#1{#1}\fi
\ifx \batitle  \undefined \def \batitle#1{#1}\fi
\ifx \bjtitle  \undefined \def \bjtitle#1{#1}\fi
\ifx \bvolume  \undefined \def \bvolume#1{\textbf{#1}}\fi
\ifx \byear  \undefined \def \byear#1{#1}\fi
\ifx \bissue  \undefined \def \bissue#1{#1}\fi
\ifx \bfpage  \undefined \def \bfpage#1{#1}\fi
\ifx \blpage  \undefined \def \blpage #1{#1}\fi
\ifx \burl  \undefined \def \burl#1{\textsf{#1}}\fi
\ifx \doiurl  \undefined \def \doiurl#1{\url{https://doi.org/#1}}\fi
\ifx \betal  \undefined \def \betal{\textit{et al.}}\fi
\ifx \binstitute  \undefined \def \binstitute#1{#1}\fi
\ifx \binstitutionaled  \undefined \def \binstitutionaled#1{#1}\fi
\ifx \bctitle  \undefined \def \bctitle#1{#1}\fi
\ifx \beditor  \undefined \def \beditor#1{#1}\fi
\ifx \bpublisher  \undefined \def \bpublisher#1{#1}\fi
\ifx \bbtitle  \undefined \def \bbtitle#1{#1}\fi
\ifx \bedition  \undefined \def \bedition#1{#1}\fi
\ifx \bseriesno  \undefined \def \bseriesno#1{#1}\fi
\ifx \blocation  \undefined \def \blocation#1{#1}\fi
\ifx \bsertitle  \undefined \def \bsertitle#1{#1}\fi
\ifx \bsnm \undefined \def \bsnm#1{#1}\fi
\ifx \bsuffix \undefined \def \bsuffix#1{#1}\fi
\ifx \bparticle \undefined \def \bparticle#1{#1}\fi
\ifx \barticle \undefined \def \barticle#1{#1}\fi
\bibcommenthead
\ifx \bconfdate \undefined \def \bconfdate #1{#1}\fi
\ifx \botherref \undefined \def \botherref #1{#1}\fi
\ifx \url \undefined \def \url#1{\textsf{#1}}\fi
\ifx \bchapter \undefined \def \bchapter#1{#1}\fi
\ifx \bbook \undefined \def \bbook#1{#1}\fi
\ifx \bcomment \undefined \def \bcomment#1{#1}\fi
\ifx \oauthor \undefined \def \oauthor#1{#1}\fi
\ifx \citeauthoryear \undefined \def \citeauthoryear#1{#1}\fi
\ifx \endbibitem  \undefined \def \endbibitem {}\fi
\ifx \bconflocation  \undefined \def \bconflocation#1{#1}\fi
\ifx \arxivurl  \undefined \def \arxivurl#1{\textsf{#1}}\fi
\csname PreBibitemsHook\endcsname

\bibitem[\protect\citeauthoryear{Boccardo et~al.}{2024}]{boc24}
\begin{barticle}
\bauthor{\bsnm{Boccardo}, \binits{L.}},
\bauthor{\bsnm{Buccheri}, \binits{S.}},
\bauthor{\bsnm{{Rita Cirmi}}, \binits{G.}}:
\batitle{Elliptic problems with superlinear convection terms}.
\bjtitle{Journal of Differential Equations}
\bvolume{406},
\bfpage{276}--\blpage{301}
(\byear{2024})
\doiurl{10.1016/j.jde.2024.06.014}
\end{barticle}
\endbibitem

\bibitem[\protect\citeauthoryear{Boccardo and Croce}{2013}]{boc13}
\begin{bbook}
\bauthor{\bsnm{Boccardo}, \binits{L.}},
\bauthor{\bsnm{Croce}, \binits{G.}}:
\bbtitle{Elliptic Partial Differential Equations Existence and Regularity of
  Distributional Solutions}.
\bpublisher{De Gruyter},
\blocation{Berlin, Boston}
(\byear{2013}).
\doiurl{10.1515/9783110315424} .
\burl{https://doi.org/10.1515/9783110315424}
\end{bbook}
\endbibitem

\bibitem[\protect\citeauthoryear{Boccardo and Gallouët}{1989}]{boc89}
\begin{barticle}
\bauthor{\bsnm{Boccardo}, \binits{L.}},
\bauthor{\bsnm{Gallouët}, \binits{T.}}:
\batitle{Non-linear elliptic and parabolic equations involving measure data}.
\bjtitle{Journal of Functional Analysis}
\bvolume{87}(\bissue{1}),
\bfpage{149}--\blpage{169}
(\byear{1989})
\doiurl{10.1016/0022-1236(89)90005-0}
\end{barticle}
\endbibitem

\bibitem[\protect\citeauthoryear{Boccardo and Gallouet}{1992}]{boc92}
\begin{barticle}
\bauthor{\bsnm{Boccardo}, \binits{I.}},
\bauthor{\bsnm{Gallouet}, \binits{T.}}:
\batitle{Nonlinear elliptic equations with right hand side measures}.
\bjtitle{Communications in Partial Differential Equations}
\bvolume{17}(\bissue{3-4}),
\bfpage{189}--\blpage{258}
(\byear{1992})
\doiurl{10.1080/03605309208820857}
\end{barticle}
\endbibitem

\bibitem[\protect\citeauthoryear{Boccardo}{2009}]{boc09}
\begin{barticle}
\bauthor{\bsnm{Boccardo}, \binits{L.}}:
\batitle{Some developments on dirichlet problems with discontinuous
  coefficients}.
\bjtitle{Bollettino dell'Unione Matematica Italiana}
\bvolume{2}(\bissue{1}),
\bfpage{285}--\blpage{297}
(\byear{2009})
\end{barticle}
\endbibitem

\bibitem[\protect\citeauthoryear{Gilbarg and Trudinger}{2001}]{trudinger}
\begin{bbook}
\bauthor{\bsnm{Gilbarg}, \binits{D.}},
\bauthor{\bsnm{Trudinger}, \binits{N.S.}}:
\bbtitle{Elliptic Partial Differential Equations of Second Order},
\bedition{3}rd edn.
\bsertitle{Classics in mathematics},
p. \bfpage{518}.
\bpublisher{Springer}, \blocation{???}
(\byear{2001}).
\doiurl{10.1007/978-3-642-61798-0}
\end{bbook}
\endbibitem

\bibitem[\protect\citeauthoryear{Hillen and Painter}{2009}]{hil09}
\begin{barticle}
\bauthor{\bsnm{Hillen}, \binits{T.}},
\bauthor{\bsnm{Painter}, \binits{K.J.}}:
\batitle{A user’s guide to pde models for chemotaxis}.
\bjtitle{Journal of Mathematical Biology}
\bvolume{58},
\bfpage{183}--\blpage{217}
(\byear{2009})
\doiurl{10.1007/s00285-008-0201-3}
\end{barticle}
\endbibitem

\bibitem[\protect\citeauthoryear{Iuorio et~al.}{2022}]{iuorio}
\begin{barticle}
\bauthor{\bsnm{Iuorio}, \binits{A.}},
\bauthor{\bsnm{Jankowiak}, \binits{G.}},
\bauthor{\bsnm{Szmolyan}, \binits{P.}},
\bauthor{\bsnm{Wolfram}, \binits{M.-T.}}:
\batitle{A pde model for unidirectional flows: Stationary profiles and
  asymptotic behaviour}.
\bjtitle{Journal of Mathematical Analysis and Applications}
\bvolume{510}(\bissue{2}),
\bfpage{126018}
(\byear{2022})
\doiurl{10.1016/j.jmaa.2022.126018}
\end{barticle}
\endbibitem

\bibitem[\protect\citeauthoryear{Stampacchia}{1965}]{sta65}
\begin{barticle}
\bauthor{\bsnm{Stampacchia}, \binits{G.}}:
\batitle{Le problème de dirichlet pour les équations elliptiques du second
  ordre à coefficients discontinus}.
\bjtitle{Annales de l'Institut Fourier}
\bvolume{15}(\bissue{1}),
\bfpage{189}--\blpage{257}
(\byear{1965})
\doiurl{10.5802/aif.204}
\end{barticle}
\endbibitem

\end{thebibliography}
\end{document}